\documentclass[11pt]{amsart}

\usepackage{amsmath, amsthm, amssymb}

\usepackage{palatino}         % text font
\usepackage[abs]{overpic}		  % overpic automatically loads graphicx

\usepackage{xcolor}
\definecolor{indigo}{rgb}{0.29, 0.0, 0.51}  % custom colors
\usepackage[colorlinks, urlcolor=indigo, linkcolor=indigo, citecolor=indigo]{hyperref}

\theoremstyle{plain}
\newtheorem{theorem}{Theorem}
\newtheorem{corollary}[theorem]{Corollary}
\newtheorem{proposition}[theorem]{Proposition}

\newtheorem{question}[theorem]{Question}

\newtheorem{conjecture}[theorem]{Conjecture}

\theoremstyle{definition}

\theoremstyle{remark}
\newtheorem{remark}[theorem]{Remark}

\numberwithin{theorem}{section}

\newcommand{\dfn}[1]{{\em #1}}        % definition
\newcommand{\R}{\mathbb{R}}           % the real numbers
\makeatletter
\newcommand*\bigcdot{\mathpalette\bigcdot@{0.6}}
\newcommand*\bigcdot@[2]{\mathbin{\vcenter{\hbox{\scalebox{#2}{$\m@th#1\bullet$}}}}}
\makeatother

\DeclareMathOperator\tb{tb}                   % Thurston-Bennequin
\DeclareMathOperator\tbb{\overline {\tb}}     % maximum Thurston-Bennequin
\DeclareMathOperator\rot{rot}                 % rotation
\DeclareMathOperator\self{sl}                 % self linking
\begin{document}

% title
\title{Knotted solid tori in contact manifolds} 

% author information
\author{John Etnyre}
\address{School of Mathematics \\ Georgia Institute
of Technology \\  Atlanta  \\ Georgia}
\email{etnyre@math.gatech.edu}

\author{Youlin Li}
\address{School of Mathematical Sciences, Shanghai Jiao Tong University, Shanghai, China}
\email{liyoulin@sjtu.edu.cn}

\author{B\"{u}lent Tosun}
\address{Department of Mathematics\\ University of Alabama\\Tuscaloosa\\Alabama}
\email{btosun@ua.edu}

%\subjclass[2020]{57R17}

%abstract
\begin{abstract}
In this note we study solid tori in contact manifolds. Specifically, we study the width of a knot type and give criteria for when it is equal to the maximal Thurston-Bennequin invariant, and when it is larger. We also prove there are many ``non-thickenable" tori in many knot types. These had previously only been observed in $S^3$. 
\end{abstract}

\maketitle
%\tableofcontents

%%%%%%%%%%%%%%%%%%%%%%%%%%%%%%%%%%%%
\section{Introduction}
%%%%%%%%%%%%%%%%%%%%%%%%%%%%%%%%%%%%

Understanding solid tori in contact manifolds has been an essential tool for our understanding of contact manifolds and Legendrian and transverse knots. For example, they are essential in the study of Legendrian and transverse cables \cite{EtnyreHonda05, EtnyreLafountainTosun12} --- and were responsible for the construction of some of the first transversely non-simple knots --- and more generally satellite knots \cite{EtnyreVertesi18}. More recently, they have appeared in the study of ``Legendrian large cables" \cite{McCullough23, Yasui16} and the classification of tight contact structures on some small Seifert fibered spaces \cite{EtnyreMinTosunPre}. 

The first real study of solid tori in contact manifolds happened in \cite{EtnyreHonda05} where the property of being uniformly thick was defined. A knot type $\mathcal{K}$ is called \dfn{uniformly thick} if any solid torus with core realizing the knot type $\mathcal{K}$ can be contained in another solid torus in the same knot type that is a standard neighborhood of a maximal Thurston-Bennequin Legendrian representative of the knot. If a knot type is uniformly thick, then it is easy to study Legendrian representatives of cables \cite{ChakrabortyEtnyreMin24, EtnyreHonda05} and more general satellites \cite{EtnyreVertesi18} of the knot. It is known that negative torus knots and the figure eight knot are uniformly thick \cite{EtnyreHonda01b,EtnyreHonda05} as well as many iterated cables of torus knots \cite{LaFountain10, LaFountain11}. Positive torus knots are not uniformly thick, and it is due to this that many of their cables are not transversely simple \cite{EtnyreHonda05, EtnyreLafountainTosun12}. We note that the contact structure on a standard neighborhood of a Legendrian knot is universally tight. Thus if $K$ is uniformly thick, the contact structure on any solid tori in this knot type must be universally tight. 

A knot type can fail to be uniformly thick in one of two ways. First, there can be solid tori with convex boundary having dividing slope larger than the maximal Thurston-Bennequin invariant, and the second way is that there can be solid tori with dividing slope less than the maximal Thurston-Bennequin invariant, but that does not thicken to a neighborhood of a maximal Thurston-Bennequin invariant knot. These latter tori are called \dfn{non-thickenable tori}. To study the first failure of uniform thickness, the first author and Honda \cite{EtnyreHonda05} defined the \dfn{width} of a knot type to be \footnote{The slope convention in this paper is the inverse of the one in \cite{EtnyreHonda05} and other main references. The convention in this paper agrees with the convention used by topologists for specifying cables and surgery slopes.}
\[
w(K)=\sup\{\text{slope of } \Gamma_{\partial S}\}
\]
where the supremum is taken over all convex solid tori $S$ in the knot type $K$ and where $\Gamma_T$ denotes the dividing curves on a convex torus $T$. It is easy to see that 
\[
w(K)\in[\tbb(K),\tbb(K)+1].
\]
So in trying to verify a knot is uniformly thick one usually first tries to show that $w(K)=\tbb(K)$. Unfortunately, we do not know that many examples of knots where we can verify this. Our first result gives a large family of such examples. 

\begin{theorem}\label{lspace}
If  $K$ is an $L$-space knot and $\tbb(K)=2g(K)-1$, where $g(K)$ is the minimal genus of a Seifert surface for $K$, 
then the width of $K$ is equal to the maximal Thurston-Bennequin invariant of $K$:
\[
w(K)=\tbb(K).
\]
\end{theorem}
It is conjectured \cite{LidmanSivek16} that for any $L$-space knot $K$ we have $\tbb(K)=2g(K)-1$. If true, then we know the width of any $L$-space knot, but we have many families of knots for which we know this is true, such as positive knots (that is a knot with a diagram in which all crossings are right-handed), the Berge knots (conjecturally the only knots with lens space surgeries), and if the equality is true for an $L$-space knot $K$ then it also holds for any cable of $K$ which is also an $L$-space knot \cite{LidmanSivek16}. 

We recall that by the $(p,q)$-cable of a knot $K$, denoted by $K_{(p,q)}$, we mean the knot type obtained by taking the isotopy class of the $(p,q)$-torus knot on the boundary of a tubular neighborhood of $K$, where we assume $p$ denotes the number of times that the knot winds along the longitudinal direction and $q$ the number of times along the meridional direction.%\footnote{this convention too differs from earlier ``contact topological" convention for cabling a knot type.}.  

We can show that having width equal to the maximal Thurston-Bennequin invariant is a common feature of cables. 
\begin{theorem}\label{cableswetbb}
Suppose $q/p<w(K)$ 
%(\textcolor{red}{this will be $p/q<w(K)$. In other words we are considering ``sufficiently negative cables"}) 
for some knot type $K$ that is not Lagrangian slice (that is there is no Legendrian representative of $K$ that bounds a Lagrangian disk). Then the $(p,q)$-cable $K_{p,q}$ satisfies 
\[
w(K_{p,q})=\tbb(K_{p,q}).
\]
The same is true if $K$ is Lagrangian slice and if $q/p\leq \min\{w(K),-1/2\}$.
\end{theorem}
In this theorem, we note that if $K$ is Lagrangian slice then $\tbb(K)=-1$ so $w(K)\in [-1, 0]$.

\begin{corollary}
If $K$ is an L-space knot and $\overline{tb}(K)=2g(K)-1$, then the $(p,q)$-cable $K_{p,q}$ satisfies $$w(K_{p,q})=\overline{tb}(K_{p,q}).$$
\end{corollary}

\begin{proof}
Since $K$ is an L-space knot and satisfies $\overline{tb}(K)=2g(K)-1$, it follows from Theorem~\ref{lspace} that $w(K)=2g(K)-1$. Additionally, $K$ is not Lagrangian slice. By Theorem~\ref{cableswetbb}, if $q/p<2g(K)-1$, then $w(K_{p,q})=\tbb(K_{p,q}).$ 

Since $K$ is an L-space knot, if $q/p\geq 2g(K)-1$, then $K_{p,q}$ is also an L-space knot \cite[Theorem 1.10]{Hedden09}, and $\tbb(K_{p,q})=2g(K_{p,q})-1$ \cite[Proposition 3.6]{LidmanSivek16}. Thus, by Theorem~\ref{lspace}, if $q/p\geq 2g(K)-1$, then $w(K_{p,q})=\overline{tb}(K_{p,q})$. 
\end{proof}

%{\color{red} can we find more examples where $w=\tbb$? I think it is true for positive cables of positive torus knots, but most of these are not $L$-space knots...}

We now study a large family of knots where $w(K)>\tbb(K)$. 
It has long been thought that Legendrian representatives of a $(p,q)$-cabled knot had their Thurston-Bennequin invariant bounded above by $pq$, \cite{LidmanSivek16}; however, in \cite{Yasui16} Yasui showed that this was not the case. This motivated the definition of {\it Legendrian large cables} \cite[Definition~$1.2.$]{McCullough23}. We say a non-trivial Legendrian cable $L\in\mathcal{L}(K_{p,q})$ is large if $tb(L)>pq$.  Later McCullough \cite{McCullough23} gave many more examples of Legendrian large cables and showed that if $K$ admits Legendrian large cables then it also has solid tori representatives on which the contact structure is virtually overtwisted. In particular, these knot types cannot be uniformly thick!  Further work of Chakraborty, Min, and the first author \cite{ChakrabortyEtnyreMin24} showed that we actually have $w(K)>\tbb(K)$ for a knot type admitting Legendrian large cables. It is interesting to note that all the examples of knots $K$ with their $(p,q)$-cables having Legendrian representatives with $\tb>pq$ were Lagrangian slice, and the cables all happened to be $(n,-1)$-cables. It seemed likely, at the time, that this was just due to the techniques used to construct the cables. But our next theorem shows that these features are actually essential. 

\begin{theorem}\label{llc}
Suppose $K$ is a knot type in $S^3$. If $\mathcal{L}(K_{p,q})$ contains a Legendrian knot with Thurston-Bennequin invariant larger than $pq$ then
\begin{enumerate}
\item $K$ has a Legendrian representative that bounds a Lagrangian disk in $(B^4, \omega_{std})$, and
\item $(p,q)=(n,-1)$ for some positive integer $n$. 
\end{enumerate}
Moreover, we note that if $K$ has a Legendrian large $(n,-1)$-cable having maximal Thurston-Bennequin invariant $-n+k$ then 
\[
w(K)\geq - \frac 1{n+k}>-1=\tbb(K).
\] 
\end{theorem}
We note that in the theorem above $k$ must be less than $n$ according to \cite{ChakrabortyEtnyreMin24}. 
%{\color{red} Try to find examples to show there are knots satisfying above, but don't have LLCs.}
Given the above results, we make the following ambitious conjecture. 

\begin{conjecture}
%If the knot type $K$ admits a Lagrangian slice representative then $w(K)>\tbb(K)$. 
%\begin{conjecture}
%If $w(K)=\tbb(K)+1$ then $K$ is the unknot. 
%\end{conjecture}
%If both conjectures are true then we know 
For a knot type $K$ in $(S^3,\xi_{std})$ we have
\[
w(K)\begin{cases}
=\tbb(K) &\text{$K$ is not Lagrangian slice}\\
=\tbb(K)+1 & \text{$K$ is the unknot}\\
\in (\tbb(K),\tbb(K)+1) & \text{otherwise}.
\end{cases}
\]
\end{conjecture}
All the examples above that support the first part of this conjecture are fibered knots, but there are non-fibered knots that also support the conjecture. For example, in \cite{GeorgeMyers20}, George and Mayers showed that all non-trivial positive twist knots are uniformly thick and most of these are not fibered. In upcoming work of Chatterjee, Min, Rodewald, and the first author, the same result will be proven for negative twist knots \cite{ChatterjeeEtnyreMinRodewaldPre}. 

It is interesting to try to understand a lower bound on the width, and hence also the maximal Thurston-Bennequin invariant, of a knot type $K$. There is little known in general, but if a fibered knot $K$ supports a contact structure $\xi_K$ (in the sense of Giroux) then it is easy to see from the construction of $\xi_K$ that $K$ admits Legendrian representatives with $\tb=-1$. However, in all known examples, $\tbb(K)>0$. While we cannot establish this in general, there is some partial progress in this direction (this seems to be known to experts in the area but otherwise not well-known and not in the literature. See also \cite{LiWanPre}).
\begin{theorem}\label{lowerbound}
If $K$ is a non-trivial fibered knot in a manifold $M$ and supports the contact structure $\xi_K$, then $\tbb(K)\geq 0$. 
\end{theorem}
We note that in the proof of the above theorem, we cannot directly use the construction of $\xi_K$ from the fibered knot to construct the desired Legendrian knot, so it is not, in general, easy to see how to create Legendrian representatives of $K$ with positive Thurston-Bennequin invariant. Based on known examples, we make the following conjecture.
\begin{conjecture}
If $K$ supports the contact structure $\xi_K$, then in $\xi_K$ there are Legendrian representatives of $K$ that have Thurston-Bennequin invariants greater than $0$. 
\end{conjecture}
We note that the converse of this conjecture is not true as there are many fibered knots in $(S^3,\xi_{std})$ that do not support $\xi_{std}$ but have positive Thurston-Bennequin invariants. For example, many connect sums of positive and negative torus knots \cite{EtnyreHonda03}, and many cables of positive torus knots \cite{LaFountain11} have this property. However, all of the known examples are fibered knots whose monodromy is reducible. So we make the following conjecture. 
\begin{conjecture}
If $K$ is a fibered knot in $S^3$ and $K$ has irreducible monodromy and does not support the standard contact structure $\xi_{std})$, then $\tbb(K)\leq -1$. 
\end{conjecture}

As mentioned above, if a knot type is uniformly thick, then any solid torus in the knot type is universally tight. 
%When studying Legendrian representatives of satellite knots and surgery on Legendrian knots, it is helpful to know when a knot is uniformly thick. Recall this means that any solid tours with convex boundary in this knot type, can be contained in a standard neighborhood of a maximal Thurston-Bennequin invariant realization of the knot. Notice that this implies that the contact structure on any solid torus in this knot type is universally tight. 
It turns out that this is a ``generic" property of solid tori realizing a knot, even for non-uniformly thick knots. More specifically, we have the following result. 

\begin{theorem}\label{vottori}
Let $S$ be any solid torus with convex boundary having $2$ dividing curves that is embedded in $(S^3,\xi_{std})$. Then $\xi_{std}$ restricted to $S$ is universally tight, unless the core of $S$ is in the knot type of a Lagrangian slice knot and the dividing curves on $\partial S$ have slope in $(-1/2, 0)$. 
\end{theorem}

%\begin{theorem}
%If $K$ is a knot type in $S^3$ that does not admit a Lagrangian slice representative, then the width of $K$ is equal to the maximal Thurston-Bennequin invariant of $K$:
%\[
%w(K)=\tbb(K).
%\]
%\end{theorem}
%{\color{red} proof: surgery on cable. NO, this is not contact $(r)$ surgery for $r>0$.}

We now turn to the other way a knot type can fail to be uniformly thick, that is, when they admit non-thickenable tori. We currently have restricted families of examples of such tori. For example, Honda and the first author \cite{EtnyreHonda05} together with LaFountain, the first and the third author \cite{EtnyreLafountainTosun12} showed that positive torus knots admit non-thickenable tori.  This was improved by LaFountain \cite{LaFountain11} to show that all iterated positive cables of positive torus knots admit such tori. The only known examples outside of the standard tight contact structure on $S^3$ occur in overtwisted contact structures on $S^3$. More specifically, Min, Mukherjee, and the first author \cite{EtnyreMinMukherjeePre} showed that negative torus knots admit non-thickenable solid tori in some overtwisted contact structures; see below for more discussion of this. Here we give a large family of new examples in many contact manifolds. 
\begin{theorem}\label{non-thickenableex}
If $K$ is a genus 1 fibered knot in the manifold $M$ or a knot with trivial monodromy, and $\xi$ is the contact structure that $K$ supports, then the knot type $K$ admits non-thickenable tori in the contact structure $\xi$.
\end{theorem}

 An immediate consequence of this (see Corollary~\ref{genus1c} for the precise statement) will be that if $\xi_K$ is tight, then $\tbb(K)=1$ for any genus 1 fibered knot $K$. 

The theorem above inspires the following conjecture. 
\begin{conjecture}
A fibered knot $K$ in a manifold $M$ admits non-thickenable tori in a contact structure $\xi$ on $M$ if and only if $K$ supports $\xi$. 
\end{conjecture}
We note work of Lafountain \cite{LaFountain11}, coupled with work of Baker, Van Horn-Morris, and the first author on cabling open books \cite{BakerEtnyreVanHorn-Morris12}, says an iterated torus knot has non-thickenable tori if and only if it supports the standard tight contact structure on $S^3$. Moreover, \cite{EtnyreMinMukherjeePre} shows that a torus knot admits a non-thickenable torus if and only if it is in the contact structure supported by that torus knot. (Here, we note that one needs to be careful about the definition of non-thickenable torus in overtwisted contact structures because there are many examples of a Legendrian knot that does not destabilize but does not have maximal Thurston-Bennequin invariant. So in this case, by non-thickenable, we mean that there is a solid torus with convex boundary having non-integral slope but is not contained in a solid torus with a different dividing slope.)

The discussion above is about non-thickenability of fibered knots, and the fact that they support a contact structure was essential. For non-fibered knots, we have no such consideration, but in upcoming work of Min and the first two authors \cite{EtnyreLiMinPre} it will be shown that if surgery on a knot $K$ yields a manifold with an incompressible torus non-trivially intersecting the surgery dual knot, then that knot type admits non-tickenable tori in some contact structure. 
%It is unclear what to conjecture about non-thickenability for non-fibered knots, but the examples of twist knots discussed above indicate that they might never have non-thickenable tori. 
So we ask the following.
\begin{question}
When do non-fibered knots never have non-thickenable tori?
\end{question}

We end with one question about which we know nothing.
\begin{question}
Is there a knot type $K$ admitting a non-thickenable torus with dividing slope larger than $\tbb(K)$?
\end{question}
Clearly from Theorem~\ref{llc} one might expect knot types $K$ with Legendrian large cables to have such tori, but it is unclear if the slope $w(K)$ is realized by a solid torus in the knot type or if there are other non-thickenable tori in this knot type.

\subsection*{Acknowledgments}
	The first author was partially supported by the National Science Foundation grant DMS-2203312 and the Georgia Institute of Technology Elaine M. Hubbard Distinguished Faculty Award. The second author was partially supported by Grants No. 12271349 of the National Natural Science Foundation of China. The third author was supported in part by grants from National Science Foundation (DMS-2105525 and CAREER DMS 2144363) and the Simons Foundation (636841, BT and 2023 Simons Fellowship). He also acknowledges the support by the Charles Simonyi Endowment at the Institute for Advanced Study where this project was started.

%%%%%%%%%%%%%%%%%%%%%%%%%%%%%%%%%%%%
\section{Background}
%%%%%%%%%%%%%%%%%%%%%%%%%%%%%%%%%%%%
We will be using standard facts about the Farey graph and from convex surface theory. These facts and the conventions used in this paper are discussed thoroughly in \cite{EtnyreRoy21}. For convenience, we recall some basic facts about the classification of contact structures on solid tori and the main result from \cite{ChristianMenkePre}.
%------------------------------------------------------------------------
\subsection{Contact structures on solid tori}
%------------------------------------------------------------------------
We recall the classification of tight contact structures on solid tori from  \cite{Giroux00, Honda00a}. We will describe solid tori as follows. Consider $T^2\times[0,1]$. Let $S_m$ be the quotient of $T^2\times [0,1]$ obtained by collapsing the leaves of a linear foliation on $T^2\times \{0\}$ of slope $m$. We say $S_m$ is a \dfn{solid torus with lower meridian $m$}. Similarly, $S^m$ is the quotient of $T^2\times [0,1]$ obtained by collapsing the leaves of a linear foliation on $T^2\times \{1\}$ of slope $m$ and we say $S^m$ is a \dfn{solid torus with upper meridian $m$}. Of course, both $S_m$ and $S^m$ are solid tori, but the first has its boundary oriented with the outward normal first, and the second has it oriented with the inward normal first. This distinction will be useful when considering contact structures. 

A path in the Farey graph is called a continued fraction block if after a change of coordinates the vertices in the path are $-1, -2, \ldots, -n$. See \cite{EtnyreRoy21} for other characterizations of continued faction block. It is not hard to see that any minimal path in the Farey graph can be broken into continued fraction blocks. 

Given slopes $m$ and $r$ we can consider a clockwise minimal path in the Farey graph from $m$ to $r$. We say this path is \dfn{partially decorated} if all the edges have a sign except the first edge. We say two partially decorated paths are the same up to \dfn{shuffling in continued fraction blocks} if the number of $+$ signs in each continued fraction bock is the same (this, of course, implies that the number of $-$ signs is the same too). 
\begin{theorem}[Giroux \cite{Giroux00}, Honda \cite{Honda00a}, 2000]\label{class}
Tight contact structure on the solid torus $S_m$ with lower meridian $m$ and convex boundary with two dividing curves of slope $r$ are in one-to-one correspondence with partially decorated paths in the Farey graph up to shuffling in continued fraction blocks. 

Moreover, a contact structure as above on $S_m$ is universally tight if and only if all the signs on the path describing the contact structure are the same. 
\end{theorem}
There is an analogous theorem for solid tori with upper meridian $m$ except that the path will run from $r$ clockwise to $m$ and all edges will have a sign except for the last edge.

%------------------------------------------------------------------------
\subsection{Splitting symplectic fillings}
%------------------------------------------------------------------------
Given a torus $T$ in a $3$-manifold $M$ with a choice of basis for $H_1(T)$ so that simple closed curves can be identified with rational numbers, we define the $s$-splitting of $M$ along $T$ to be the result of cutting $M$ along $T$ and gluing two solid tori to the resulting boundary components with each solid torus glued so that the meridian is glued to the curve of slope $s$. 

If we have a contact structure $\xi$ on $M$, we say a torus $T$ is a \dfn{mixed torus} if $T$ has a neighborhood $N=T^2\times [-1,1]$ such that $T=T^2\times\{0\}$ and $T^2\times[-1,0]$ and $T^2\times[0,1]$ are basic slices with opposite signs. Notice that to talk about the signs of the basic slices $T$ must be oriented and this orients $T^2\times[-1,0]$ and $T^2\times[0,1]$. Let $s_i$ be the dividing slope on $T^2\times\{i\}$ for $i=-1,0,1$ and $E_T$ be the set of vertices in the Farey graph in the interval $[s_1, s_{-1}]$ with an edge to $s_0$. For $e\in E_T$ we say an $e$-splitting of $(M,\xi)$ is obtained by gluing solid tori to $M\setminus T$ as above and extending the contact structure by the unique tight contact structure on these solid tori. (The torus $T$ is oriented so that we can discuss the signs on the basic slices and one of the boundary components of $M\setminus T$ will be oriented as $T$ was and the other will have the opposite orientation. To the first, we glue a solid torus with upper meridian $e$, and to the second, we glue a solid torus with lower meridian $e$. As these solid tori have unique tight contact structures, this distinction is not essential here, but in discussions below, it will be useful to keep in mind.) We note that each of these solid tori that is added to the $e$-splitting of $M$ is a standard neighborhood of a Legendrian knot (since the dividing curves on their boundary are longitudinal for their meridian). 

\begin{theorem}[Christian-Menke \cite{ChristianMenkePre}]\label{splitting}
If $X$ is an exact symplectic filling of $(M,\xi)$ and $T$ is a mixed torus in $(M,\xi)$, then there is a symplectic filling $X'$ of the $e$-splitting of $(M,\xi)$ such that $X$ is recovered by attaching a round $1$-handle to the Legendrian knots determined by the $e$-splitting. 
\end{theorem}

Recall that a round handle attached to $X$ along knots $K_1$ and $K_2$ is the result of attaching a $1$-handle with attaching sphere a point on $K_1$ and a point on $K_2$ and then attaching a $2$-handle to $K$ where $K$ is obtained by connected summing $K_1$ and $K_2$ in the $3$-manifold boundary of the result of the $1$-handle attachment. One must specify framings on $K_1$ and $K_2$ that induce a framing on $K$. In the contact setting, the $K_i$ have contact framings, and $K$ is the connected sum of Legendrian knots and has a natural framing, and the handle is attached as a Stein $2$-handle to the Legendrian knot $K$. 

%%%%%%%%%%%%%%%%%%%%%%%%%%%%%%%%%%%%
\section{Width of knots}
%%%%%%%%%%%%%%%%%%%%%%%%%%%%%%%%%%%%
In this section, we prove our two theorems about the width of certain knots. We begin with the result stating that $w(K)=\tbb(K)$ if $K$ is an $L$-space knot with $\tbb(K)=g(K)-1$.  
\begin{proof}[Proof of Theorem~\ref{lspace}]
Let $K$ be an $L$-space knot with $\tbb(K)=2g(K)-1$. Suppose that $w(K)>\tbb(K)$. Then there will be a solid torus with convex boundary having dividing slope $p/q$ for some $p/q\in (\tbb(K),\tbb(K)+1)$. We can assume that the boundary of this torus is in standard form and let $L$ be a Legendrian divide on the torus. Clearly, $L$ is a Legendrian knot in the knot type of the $(q,p)$-cable of $K$. The contact framing on $L$ agrees with the framing coming from the torus, and this framing differs from the Seifert framing by $pq$. Thus $\tb(L)=pq$. Let $M$ be the result of Legendrian surgery on $L$. One can build a Stein filling $X$ of $M$ by attaching a Stein $2$-handle to $B^4$ along $L$. It is well-known, see \cite[Corollary~7.3]{Gordon83}, that smoothly $M$ is also the result of $(pq-1)/q^2$ surgery on $K$. 

We note that the result of smooth $r$ surgery on $K$ will be an $L$-space if $r\geq 2g(K)-1$, \cite{Hedden10, LiscaStipsicz04}. Moreover, any symplectic filling of an $L$-space must be negative definite, \cite{OzsvathSzabo04a}.  We claim that  $(pq-1)/q^2>2g(K)-1$. Indeed, if $p/q\in (n,n+1)$ then $p=nq+r$ for $1\leq r<q$. Now one may easily check that $r/q-1/q^2\in (0,1)$ from which we see that $(pq-1)/q^2=n + r/q - 1/q^2$ and hence is also in $(n,n+1)$. Since $(pq-1)/q^2>2g(K)-1$ we see that $M$ is an $L$-space and thus any symplectic filling must be negative definitely; however, $X$ is positive definite since $pq-1>0$. This contradiction proves that $w(K)=\tbb(K)$. 
\end{proof}

We now prove Theorem~\ref{cableswetbb} that claims the width equals the maximal Thurston-Bennequin invariant for cables of non-Lagrangian slice knots if the cable coefficient is sufficiently negative. We note that this proof is modeled on the proof of Theorem~1.2 in \cite{EtnyreHonda05}.

\begin{proof}[Proof of Theorem~\ref{cableswetbb}]
Suppose $p/q<w(K)$ and $K$ is not Lagrangian slice. By hypothesis, there is a solid torus in the knot type $K$ with convex boundary having dividing slope larger than $q/p$, and by the classification of tight contact structures on solid tori, we know that inside of this neighborhood, we can find a convex torus with dividing slope $q/p$. Let $L$ be a Legendrian divide on this torus. This is smoothly a $(p,q)$-cable of $K$, and the contact planes have twisting $0$ relative to the torus, thus $\tb(L)=pq$. We will prove that the width of $K_{p,q}$ is also $pq$. This will establish that $\tbb(K_{p,q})=pq$. 

Arguing by contradiction, we assume that there is a solid torus $S$ in the knot type $K_{p,q}$ with convex boundary having dividing slope larger than $pq$. From now on we will use coordinate on $\partial S$ where the longitude comes from the cabling torus. In these coordinates, the slope of the dividing curves on $\partial S$ are larger than $0$. By shrinking $S$ if necessary, we can assume that $\partial S$ has $2$ dividing curves of slope $1/n$ for some large positive integer $n$. 

Let $A$ be an annulus in the complement of $S$ such that the complement of $S$ together with a neighborhood of $A$ is a solid torus $N(K)$ in the knot type of $K$. We can make $A$ convex with boundary ruling curves on $\partial S$. Notice that there are two dividing curves on $A$ and they must run from one boundary component to the other, since if not we could Legendrian realize a core curve on $A$ that was disjoint from the dividing curves and then expand $S$ to $S'$ such that $\partial S'$ contained this Legendrian. This solid torus $S'$ can be made to have convex boundary, and its dividing slope must be $0$. This implies that in $S'\setminus S$, there is a convex torus of slope $\infty$, but a Legendrian divide on this torus bounds a disk in $S$, which implies that the contact structure is overtwisted. Thus, the dividing curves on $A$ must go from one boundary component to the other. 

Consider $S\cup (A\times [-1,1])$ where the contact structure on $A\times [-1,1]$ is invariant under translation in the $[-1,1]$-direction. After rounding corners, we note that this is diffeomorphic to $T^2\times [-1,1]$, and we will use this to denote this manifold. Let $T_i$ denote $T^2\times \{i\}$ with which we consider $T_1=\partial N(K)$ and $T_{-1}\subset N(K)$. Notice that, since the slope of the dividing curve on $S$ is $\frac{1}{n}$, we obtain that the slope of the dividing curves on $T_{-1}$ is obtained by performing $n+1$ negative Dehn twists to the dividing curves on $T_1$. So, we can choose coordinates on $T^2$ so that the dividing slope on $T_1$ is $1$ and on $T_{-1}$ is $-1/n$. By the classification of tight (minimally twisting) contact structures on thickened tori, we know we can find a convex torus $T_0$ in $T^2\times [-1,1]$ that has dividing slope $0$ and this splits $T^2\times [-1,1]$ into two basic slices $T^2\times[-1,0]$ and $T^2\times [0,1]$ each of which have two possible contact structures, this fact and the following statements can be found in \cite{Honda00a}. The possible relative Euler classes for $T^2\times[-1,0]$ are $\pm \begin{bmatrix}1-n\\ 1\end{bmatrix}$ and for $T^2\times [0,1]$ are $\pm \begin{bmatrix}0\\ 1\end{bmatrix}$. So the relative Euler class of $T^2\times [-1,1]$ is simply the sum of two of these. We know that evaluating the relative Euler class on an annulus $A'$ that is a convex annulus in $A\times[-1,1]$ obtained as a core circle cross $[-1,1]$ is simply $\chi(A'_+)-\chi(A'_-)$ where $A'_\pm$ are the $\pm$ components of the complement of the dividing set on $A'$ and $\chi$ denotes Euler characteristic. Thus, the relative Euler class of $T^2\times [-1,1]$ evaluated on $A'$ is $0$, and hence the basic slices have to have opposite signs. This implies that the contact structure on $T^2\times [-1,1]$ is virtually overtwisted, but that would further imply that $K$ can be realized by a solid torus with a virtually overtwisted contact structure on it.  But this is ruled out by Theorem~\ref{vottori} since $K$ is not Lagrangian slice. Thus $w(K)=\tbb(K)$. 

If $K$ is Lagrangian slice, we obtain the same conclusion unless $q/p\in(-1/2,0)$. This is because the slope of torus with a virtually overtwisted contact structure at the end of the previous paragraph has an edge in the Farey graph to $q/p$ and so again we get a contradiction to Theorem~\ref{vottori}. 
\end{proof}

We now turn to the result concerning knot types that admit Legendrian large cables and show that such knots must be Lagrangian slice, the cable slope of the large cable must be $-\frac 1n$, and the width of the knot is larger than the maximum Thurston-Bennequin number $\tbb$. 

\begin{proof}[Proof of Theorem~\ref{llc}]
Suppose that $K$ is a knot type such that there is a Legendrian representative $L$ in $\mathcal{L}(K_{p,q})$ that has $\tb(L)=pq+k$ for some $k>0$. Then according to Theorem~1.10 in \cite{ChakrabortyEtnyreMin24} there is a continued fraction block of length $2k$ such that there are $k$ positive basic slices and $k$ negative basic slices and the slope on the middle torus is $q/p$. Thus we have a mixed torus $T$ of slope $q/p$. 

Let $e$ be an exceptional slope for $q/p$, that is $e$ is an element of $E_T$. The torus $T$ cuts $S^3$ into a solid torus $S$ and $S^3_K$, the complement of a neighborhood of $K$ in $S^3$. Thus the $e$-splitting of $S^3$ along $T$ will result in $S$ with a solid torus with meridional slope $e$ attached and $S^3_K$ with a solid torus of slope $e$-attached. The first is a lens space and the second is $e$ Dehn surgery on $K$, denoted $S^3_K(e)$. Since we know that any symplectic filling of a lens space has a connected boundary \cite{Etnyre04b,Schonenberger2007} we know that the symplectic filling $X'$ of the $e$-splitting of $M$ along $T$ coming from Theorem~\ref{splitting} has two components $X_1$ and $X_2$ with $X_1$ filling the lens space. Now the filling $B^4$ of $S^3$ is recovered from $X_1\cup X_2$ by attaching a round $1$-handle. Thus one of $X_i$ must be a homology ball and the other must be a homology $S^1\times D^3$. Since lens spaces do not have the latter filling \cite{ChristianLi23, Lisca08, EtnyreRoy21} $S^3_K(e)$ must be a homology $S^1\times D^3$ and hence $S^3_K(e)$ must be a homology $S^1\times S^2$. From this we know that $e$ must be $0$. Denote $\partial X_i$ by $M_i$. 

Since $e=0$ and $e$ has an edge in the Farey graph to $q/p$ we see that $q/p=1/n$ for some $n$. Notice if $n$ is positive, then the contact structure on $M_1$ is obtained by gluing a solid torus with lower meridian $\infty$ and dividing slope $1/n$ to a solid torus with upper meridian $0$ and dividing slope $1/n$. In the latter torus we can find a convex torus with dividing slope $\infty$ and thus a Legendrian divide on this torus will bound an overtwisted disk. Since $M_1$ is symplectically fillable, we cannot have $n$ positive. 

Now considering negative $n$, let $S'$ denote the solid torus that $T$ bounds. Since the slope of $T$ is in a continued fraction block the other tori in this block must have slope $1/m$, and since you can shuffle basic slices in the continued fraction block we can assume that all the basic slices between $T$ and the torus of slope $-1$ have the same sign. Thus the contact structure on $S'$ is universally tight. Further, let $S''$ be the solid torus in $S'$ that has convex boundary with dividing slope $-1$. Notice that $S''$ is a standard neighborhood of a Legendrian knot $L$ in the knot type of $K$. One of the components of the $e$-splitting of $S^3$ along $T$ is obtained by removing $S''$ from $S^3$ and gluing in a solid torus with tight contact structure having meridian $0$. We note that this is equivalent to removing the solid torus $S''$ and gluing in a solid torus with tight contact structure having meridian $0$. That is $M_2$ is obtained from $S^3$ by contact $(+1)$ surgery on $L$. Now from \cite{ConwayEtnyreTosun21} we know that $L$ must be Lagrangian slice and we have already seen that $q/p=1/n$. Thus completing the proof of Items~(1) and~(2) in the theorem. 

We are left to see that the width of $K$ satisfies $w(K)\geq -\frac{1}{n+k}$, but this is clear since $T$ is the center torus in a continued fraction block of length $2k$. 
\end{proof}

We now prove the estimate on the lower bound of the maximal Thurston-Bennequin invariant of a fibered knot $K$ that says if $K$ is a non-trivial fibered knot in a manifold $M$ and supports the contact structure $\xi_K$, then $\tbb(K)\geq 0$. 

\begin{proof}[Proof of Theorem~\ref{lowerbound}]
Let $\Sigma$ be the union of the closures of two fibers in the fibration of $M-K$. From the construction of $\xi_K$, it is clear that $\Sigma$ is convex with dividing set $K$. Because $K$ is non-trivial, the fiber of the fibration is of positive genus, so one may use the Legendrian realization principle \cite{Honda00a} to realize a non-separating simple closed curve $C$ on $\Sigma$ as a Legendrian curve in the complement of $K$. Then, using a local model for this Legendrian curve one can create a new convex surface $\Sigma'$ (agreeing with $\Sigma$ outside a neighborhood of the Legendrian curve) with dividing set $K$ union two parallel copies of $C$. Now we can realize a copy of $K$ as a Legendrian curve in the characteristic foliation of $\Sigma'$ that is disjoint from the dividing set. This implies that its Thurston-Bennequin invariant is $0$. 
\end{proof}

%%%%%%%%%%%%%%%%%%%%%%%%%%%%%%%%%%%%
\section{Virtually overtwisted solid tori}
%%%%%%%%%%%%%%%%%%%%%%%%%%%%%%%%%%%%
This section is devoted to the proof that the standard contact structure on $S^3$ restricted to a solid torus in the knot type $K$ will be universally tight unless $K$ is Lagrangian slice and the slope of the dividing curves is in $(-1/2,0)$. %We note that the proof is very similar to the proof of Theorem~\ref{llc}, but as there are some differences, we give it here. 
\begin{proof}[Proof of Theorem~\ref{vottori}]
Suppose $S$ is a solid torus with convex boundary having dividing slope $s$ in $(S^3,\xi_{std})$ in the knot type $K$ that is virtually overtwisted. We know from Theorem~\ref{class} that $\xi_{std}$ restricted to $S$ will be given by a minimal path in the Farey graph from $\infty$ clockwise to $s$ with all edges decorated by a $\pm$ except for the first edge. Since the contact structure is virtually overtwisted, some of the signs must be different in the path. Suppose $r$ is the slope of a torus between a $+$ edge and a $-$ edge. There will be a mixed torus $T$ in $S$ with slope $r$. We will choose $r$ to be the first place, moving clockwise from $\infty$, where the signs on the edges change. %Our arguments below will be very parallel to those in the proof of Theorem~\ref{llc} in the previous section.

We can repeat the argument in the proof of Theorem~\ref{llc} above to see that $r=1/n<0$ and $K$ admits a Legendrian representative that is Lagrangian slice.

Finally, since $r=1/n$ with $n\leq -2$ (note $-1$ would make $r$ an integer and there can be no changes in the signs on the path in the Farey graph at $-1$ since the edge from $\infty$ to $-1$ will not have a sign), and we know that $s>r$ and hence $s\in (-1/2,0)$. 
\end{proof}

%%%%%%%%%%%%%%%%%%%%%%%%%%%%%%%%%%%%
\section{Non-thickenable tori}
%%%%%%%%%%%%%%%%%%%%%%%%%%%%%%%%%%%%
In this section we consider non-thickenable tori and prove Theorem~\ref{non-thickenableex} that states for genus one open books, or open books with trivial monodromy, there are non-thickenable tori in the knot type of the binding contained in the contact structure supported by the open book. The proof will be broken into two cases. We start with genus 1 fibered knots.

\begin{proposition}\label{genus1}
Let $K$ be a genus-$1$ fibered knot in a closed $3$-manifold $M$ and $\xi_K$ be the contact structure it supports. Then there are non-thickenable solid tori $S_n^\pm$ in $(M,\xi)$ for $n\geq 2$ and $S_1$ with convex boundary having two dividing curves of slope $1/n$. 
\end{proposition}

\begin{corollary}\label{genus1c}
With the notation from Proposition~\ref{genus1}
We note that $S_1$ is a standard neighborhood of a Legendrian knot $L$ in the knot type $K$ with $\tb(L)=1$ and $\rot(L)=0$. In particular, $\tbb(K)=1$ in $\xi_K$ for all fibered knots $K$, if $\xi_K$ is tight. 

Moreover, if $\xi$ is a tight contact structure on $M$ not isotopic to $\xi_K$ then $\tbb(K)\leq 0$ in $\xi$ and if $M=S^3$ or $-K$ is isotopic to $K$ then $\tbb(K)\leq -1$. 
\end{corollary}

\begin{proof}
It is clear that $S_1$ in Proposition~\ref{genus1} is a standard neighborhood of a Legendrian knot $L$ with the claimed properties (by the Bennequin inequality). 

Now if $\xi$ is a tight contact structure on $M$ not isotopic to $\xi_K$ then we note that $K$ cannot admit a Legendrian representative $L$ with $\tb(L)=1$ since its transverse push-off would have $\self=2g(K)-1$ and the main result in \cite{EtnyreVanHornMorris10} would imply that $K$ supports $\xi$. Now suppose that there is a Legendrian knot $L$ with $\tb(L)=0$. The Bennequin inequality implies that $\rot(L)=\pm 1$. If it is $-1$, then the transverse push-off of $L$ would again have $\self=2g(K)-1$, leading to the same contradiction as above. If $\rot(L)=1$ and $-K$ is isotopic to $K$ then $\rot(-L)=-1$ and the transverse push-off of $-L$ will lead to the same contradiction. Finally, in the standard tight contact structure on $S^3$ it is well-known that if $K$ has a Legendrian knot $L$ with rotation number $r$ then there is a Legendrian knot with the same Thurston-Bennequin invariant and rotation number $-r$. 
\end{proof}

\begin{proof}[Proof of Proposition~\ref{genus1}]
We first recall a family of contact structures on torus bundles. Let $Y$ be a torus bundle over $S^1$. For such a manifold, there is a diffeomorphism $\psi:T^{2}\to T^{2}$ such that $Y$ is the mapping torus of $\psi$, that is $Y$ is $T^{2}\times[0,1]$ after $(x,1)$ is identified with $(\psi(x),0)$. We now consider $\alpha=f(t)\, d\theta+ g(t)\, d\phi$ on $T^2\times [0,1]$ where $(\theta,\phi)$ are angular coordinates on $T^2$ and $t$ is the coordinate on $[0,1]$. The $1$-form $\alpha$ defines a contact structure if and only if $f'g-g'f>0$. This requirement is the same as saying the curve $t\mapsto (f(t),g(t))$ moves clockwise around the origin in $\R^2$. Note that the contact structure is tangent to the $[0,1]$-direction and induces a linear foliation on the tori $T^2\times\{t\}$ of slope $f(t)/g(t)$. 

We can suppose that $f(0)=0$ and $g(0)=1$, so the characteristic foliation $L_0$ on $T^2\times \{0\}$ is by lines of slope $0$. Let $L_1$ be the linear foliation on $T^2\times\{1\}$ given by $\psi^{-1}(L_0)$. Now choose $f_1$ and $g_1$ so that they parameterize a curve on $\R^2$ that starts at $(0,1)$ and moves clockwise until it reaches the angle of the foliation $L_1$. With this choice, we obtain a contact structure $\xi'_1$ on $Y$. If we let $f_{n+1}$ and $g_{n+1}$ be $f_1$ and $g_1$ extend by rotating around the origin in $\R^2$ $2\pi n$ extra times we obtain a contact form $\alpha_{n+1}=f_{n+1}(t)\, d\theta + g_{n+1}(t)\, d\phi$ and the contact structure $\xi'_{n+1}$ on $Y$. 

Now suppose, that $M$ has an open book decomposition with binding $K$ and page genus~$1$. Let $Y$ be the result of $0$-surgery on $K$. Note $Y$ is a torus bundle over $S^1$. Let $K'$ be the surgery dual of $K$ in $Y$. The meridian for $K$ will give a longitude $l$ for $K'$. We can assume that the framing of $K'$ given by $l$ is the same as the product framing on $K'$ coming from $T^2\times[0,1]$. With this choice, we notice that $K'$ in $(Y,\xi'_n)$ can be realized by a Legendrian knot $L'_n$ with contact twisting $-n$. Let $N'_n$ be a standard neighborhood of $L'_n$ in $(Y,\xi'_n)$ and let $\xi^\pm_n$ be the contact structure obtained from $(Y,\xi'_n)$ by removing $N'_n$ and gluing in a solid torus $S_n^\pm$ with meridional slope $\infty$ (using the longitude-meridian coordinates on $K$). We note that $\partial S_n^\pm$ will be convex with dividing slope $1/n$. Thus when $n>1$ we see that there are two possible contact structures on $S_n^\pm$, one is denoted as $S_n^+$ and the other $S_n^-$. On the other hand, for $n=1$ there is a unique contact structure on the solid torus so $S_1^\pm$ is the same independent of the sign $\pm$, but we keep the sign just to be consistent with the other cases. 

\noindent{\bf Claim:} (1) the contact structure $\xi_n^\pm$ is $\xi_K$ for all $n$ and $\pm$, and that (2) the $S_n^\pm$ are non-thickenable tori in $(M,\xi_K)$. 

The result clearly follows from these two claims. 

To establish the first claim we notice that $\xi_n^\pm$ is obtained from $\xi_n'$ by contact $(+n)$-surgery on $L_n'$. If we let $T_n'$ be the transverse push-off of $L_n'$ then contact $(+n)$-surgery on $L_n'$ is the same as inadmissible transverse $0$-surgery on $T_n'$. Recall this means one takes a standard neighborhood of $T_n'$ and removes it and glue in a new solid torus with meridional slope $0$ (using longitude-meridian coordinates on $T_n'$ where the longitude is the meridian of $K$ and the meridian is the meridian of $T_n'$, this is the same as the $\infty$ slope using longitude-meridian coordinates coming from $K$). 

We now give an alternate description of $\xi'_n$. Consider 
\[
\alpha_{n,s,r}=s\, dt + r(f_n(t)\, d\theta + g_n(t)\, d\phi)
\]
where the functions $f_n$ and $g_n$ are defined above. If $s=0$ and $r=1$ then $\alpha_{n,s,r}=\alpha_n$. For $s\in[0,1]$ the forms $\alpha_{n,s,r}$ are all contact, and hence $\xi_n'$ is isotopic to the kernel of $\alpha_{n,1,r}$. Similarly $\alpha_{n,1,r}$ are contact forms for all $r>0$. So we may further isotope $\xi_n'$ to the kernel of $\xi_n''=\alpha_{n,1,\epsilon}$ for $\epsilon>0$ arbitrarily small. We note that $\mathcal{F}=\ker \alpha_{n,1,0}$ defines the foliation of $Y$ by fibers of the fibration of $Y$ over $S^1$. So $\ker \alpha_{n,1,r}$ shows that $\xi_n'$ is a deformation of the foliation $\mathcal{F}$. Now $K'$ can be thought of as a section of the fibration $Y\to S^1$ and as such it is transverse to $\mathcal{F}$ and hence to $\ker \alpha_{n,1,\epsilon}$. It is not hard to see that this transverse knot is transversely isotopic to $T_n'$. (Indeed notice that $T_n'$ is transverse to $\xi_n'$ and the deformations to $\ker_{n,1,\epsilon}$ always keep $T_n'$ transverse to the contact planes.) Thus $\xi_n^\pm$ is obtained from $\xi_n''$ by inadmissible transverse $0$-surgery on $T_n'$.

Let $N_n'$ be a standard neighborhood of $T_n'$ in $\xi_n''$. So $(M,\xi_n^\pm)$ is obtained from $(Y,\xi_n'')$ by removing $N_n'$ and gluing in a solid torus $\widetilde S_n^\pm$ with meridional slope $\infty$ (when measured in the longitude-meridian coordinates coming from $K$ in $M$). 
We can construct a Heegaard surface $\Sigma$ for $M$ by taking two fibers $F$ and $F'$ of the fibration of $(Y-N'_n)\to S^1$ and connecting them by an annulus in $\widetilde S_n^\pm$. All of the singularities in the characteristic foliation of $F$ are positive (since $\xi_n''$ is a small perturbation of $\mathcal{F}$) and similarly for $F'$. But to build the Heegaard surface $\Sigma$ one must reverse the orientation on $F'$. Thus the singularities in $F'\subset \Sigma$ are negative and in $F\subset \Sigma$ are positive. Moreover, the characteristic foliation on the annulus in $\Sigma$ consists of leaves running from one boundary component to the other. Thus Giroux \cite{Grioux93} tells us the dividing set on $\Sigma$ is simply $K$ sitting on $\Sigma$. We note that $\Sigma$ breaks $M$ into two handlebodies and each of these handlebodies is a subset of $(Y,\xi_n'')$ and hence is tight (since $\xi_n''$ is a perturbation of the taut foliation $\mathcal{F}$, it must be tight \cite{EliashbergThurston98}). Now since $\Sigma$ is also a union of two pages of the open book of $M$ associated to $K$, we know by \cite{Torisu00} that $\xi_n^\pm$ is supported by $K$ and hence isotopic to $\xi_K$ as claimed. 

We now prove the second claim above that the $S_n^\pm$ are non-thickenable tori in $(M,\xi_K)$. Suppose that there was a solid torus $S$ containing $S_n^\pm$ in $\xi_K$. We can then surger $S_n^\pm$ in $S$ to get back to $Y$ and the contact structure $\xi_n'$. The torus $S$ will become a torus $S'$ in $Y$ with convex boundary having dividing slope larger than $-n$. We will show that this will allow us to find another solid torus with convex boundary having dividing slope $-n+1$. This will be a neighborhood of a Legendrian knot $L$ with $\tb=-n+1$ that is smoothly isotopic to $L_n'$. But since the Giroux torsion of $\xi_n'$ is $n-1$ we know from \cite{HondaKazezMatic02} that the contact twisting of Legendrian knots isotopic to $L_n'$ is bounded above by $-n$. This contradiction proves that $S$ did not exist and $S_n^\pm$ is non-thickenable. 

We are left to show that given $S'$ in $(Y,\xi_n')$ we can find a solid torus in the same knot type with convex boundary having dividing slope $-n+1$. If the dividing slope of $S'$ is larger than $-n+1$ then we can find the desired torus inside of $S'$. So, we are left to consider the case when the dividing slope of $\partial S'$ is between $-n$ and $-n+1$. In this case, we claim there is a Legendrian $L'$ knot isotopic to $L'_n$ outside of $S'$. Given this, we note that we can take an annulus $A$ with one boundary component on $L'$ and the other a ruling curve on $\partial S'$. We can make the annulus convex and the dividing curves on $A$ will intersect $L'$ exactly $2n$ times and will intersect the dividing curves on $\partial S'$ more times (since the curves of slope $0$ will intersect any curve of slope between $-n$ and $-n+1$ more than $n$ times). Thus we can find a bypass on $A$ for $\partial S'$. Attaching the bypass will increase the dividing slope. We can continue to do this until we have a solid torus with boundary having dividing slope $-n+1$. 

We now consider the existence of $L'$. We do not actually prove $L'$ exists in $Y$, but that we can find a finite cover of $Y$ that unwraps the fibers of $Y\to S^1$ in which we can find $L'$, which is disjoint from a lift of $S'$. This is clearly sufficient to obtain the same contradiction as in the previous paragraph. We note that we consider the cover of $Y$ in which the fiber is unwrapped to $\R^2$, then there will be lifts of $L_n'$ and $S'$ that will be disjoint, but now there is clearly a finite cover in which they are disjoint too. % To see the lifting of the monodromy recall that any diffrom of $T^2$ is just an element of $SL(2,Z)$ acting on $\R^2$ modulo the $\Z^2$ action on $\R^2$. So we can just mod out by larger $\Z^2$ latticies. 
\end{proof}

We now turn to the case of an open book with trivial monodromy. 
\begin{proposition}
Let $B$ be the binding of the open book for $\#_{2g} S^1\times S^2$ with page the surface $\Sigma$ of genus $g$ with one boundary component and monodromy $\phi$, the identity map on $\Sigma$. Then in the tight contact structure $\xi_{std}$ on $\#_{2g} S^1\times S^2$, $B$ admits non-thickenable tori $S_n^\pm$ for $n>1$ and $S_1$, where $\partial S_n^\pm$ is convex with $2\gcd(2g-1,n)$ dividing curves of slope $(2g-1)/n$ and $\partial S_1$ is convex with $2$ dividing curves of slope $2g-1$. Moreover, these are the only non-tickenable tori in the knot type of $B$. 
\end{proposition}
\begin{remark}
We notice that $\xi_{std}$ on $\#_{2g} S^1\times S^2$ is supported by the open book in the proposition. 
\end{remark}
\begin{proof}
We will prove the proposition in three steps. We will first see that the only possible slopes for non-thickenable tori in the knot type of $B$ are the ones given in the proposition, we will then construct these solid tori in $\xi_{std}$, and we will finally show they are actually non-thickenable. 

\smallskip
\noindent{\bf Step 1:} Identifying the possible non-thickenable slopes. Suppose $S$ is a non-thickenable solid torus with convex boundary in the knot type $B$. There are integers $r$ and $s$ such that the boundary of $S$ has $2\gcd(r,s)$ dividing curves of slope $r/s$. (Here we use standard longitude-meridian coordinates on $\partial S$.) Now consider $M=(\#_{2g} S^1\times S^2)\setminus S$. Clearly, $M=\Sigma\times S^1$ where the $S^1$ factor is parallel to the meridian of $S$. 

We can give a framing to fibers of $M$ using the product structure on $M$. Thus, we can assign integers to the contact framing on any Legendrian knot isotopic to an $S^1$-fiber. Let $L$ be a Legendrian realization of an $S^1$-fiber on the interior of $M$ with the largest possible contact framing. We note that the contact framing of $L$ must be less than or equal to $-1$ because if it were $0$ (or if larger, we could stabilize to make it $0$), then we could thicken $S$ to a solid torus with dividing slope $\infty$ and thus we would see that $\xi_{std}$ is overtwisted which it is not. Let $-k$ be the contact framing on $L$ for some $k\geq 1$. Let $N(L)$ be a standard neighborhood of $L$, and we can choose the product structure on $M$ so that $N(L)$ is a disk $D$ on the interior of $\Sigma$ times $S^1$. Now we can think of $\Sigma$ as being built from $D$ by attaching $2g$ $1$-handles. Let $\alpha_1, \ldots, \alpha_{2g}$ be the cores of these $1$-handles. That is, the $\alpha_i$ are disjoint arcs on $\Sigma$ with endpoints on $\partial D$ and a neighborhood of $D$, and the $\alpha_i$ is $\Sigma$ minus a collar neighborhood of $\partial \Sigma$. Let $A_i=\alpha_i\times S^1$. These are annuli in the interior of $M$ with boundary on $N(L)$. We can assume that $\partial A_i$ are ruling curves on $\partial N(L)$ and that the $A_i$ are convex. We notice that the dividing curves on each $A_i$ all run from one boundary component to the other because if not, we would have a bypass for $\partial N(L)$ that would show that $L$ destabilizes. Now if we add $I$-invariant neighborhoods of the $A_i$ to $N(L)$ and round the corners, we get a manifold $M'$ such that $M-M'$ is a thickened torus, and since the dividing curves on $\partial N(L)$ have slope $-k$ we see that the dividing curves on $\partial M'$ have $2\gcd (2g-1,k)$ dividing curves of slope $-k/(1-2g)$. The longitude-meridian coordinates on $S$ are the inverse of those on $\partial M'$ coming from $N(L)$ so we see that the complement of $M'$ is a solid torus $S'$ containing $S$ with dividing slope $(2g-1)/k$. Thus, since $S$ is non-thickenable we must have that $r/s=(2g-1)/k$ and $\partial S$ must have less than or equal to $2\gcd(2g-1,k)$ dividing curves. (Recall in the definition of non-thickenable, we said that any torus containing $S$ must have the same slope as $S$ but there could be more dividing curves as one can always ``fold" a convex surface in an $I$-invariant neighborhood to increase the number of dividing curves.)

We are left to see that $\partial S$ has $2\gcd (2g-1, k)$ dividing curves. From above, we just need to see that it does not have fewer dividing curves. Suppose that $l=\gcd(2g-1,k)$ and $\partial S$ has $2l'$ dividing curves with $l'<l$. (Here we assume that $l>1$ since if not, there is nothing to prove.)

Let $\beta_1, \ldots, \beta_{2g}$ be arcs properly embedded in $\Sigma$ that cut $\Sigma$ into a disk that contains $D$ and let $B_i=\beta_i\times S^1$. Thus, $M$ cut along the $B_i$ is a solid torus $N$ that contains  $N(L)$. We can arrange that $\partial M$ has ruling curves of slope $\infty$, that is, parallel to the $S^1$-factor, and then choose the $B_i$ to have boundary ruling curves on $\partial M$ and be convex. We note that the dividing curves on each $B_i$ must run from one boundary component of $B_i$ to the other since, otherwise, there would be a bypass for $S$ that we could use to thicken $S$. If we cut $M$ along the $B_i$ and round the corners, we see that $\partial N$ is convex, and we compute the slope of the dividing curves as follows. Suppose $a$ and $b$ are positive, relatively prime integers such that $a/b=k/(2g-1)$. Thus $k=la$ and $2g-1=lb$. Now the homology class of a dividing curve on $\partial M$ is $(l' a, l'b)$. Thus, the homology class of a dividing curve on $\partial N$ is $(l'a, l'b-2g)$. Recalling the slope conventions above, we see that the dividing curves on $\partial N$ have slope $-\frac{l'a}{2g-l'b}$. Notice that this function is decreasing from 0 to $l$, and thus $-\frac{l'a}{2g-l'b}>-\frac{la}{2g-lb}=-k$. So $N$ is a thickening of $N(L)$. Below we will show that $N(L)$ cannot be thickened in $M$ and thus $l'$ cannot be less than $l$ and we will complete Step 1. 

Consider an annulus $A$ in M with one boundary component a ruling curve on $\partial M$ and the other on $\partial N(L)$. Notice that there are no boundary parallel dividing curves on $A$ since they would give a bypass for either $\partial M$ or $\partial N(L)$, and these are not allowed, as the former would thicken $S$ (which we are assuming does not thicken) while the latter would destabilize $L$ (which we are assuming does not destabilize). Thus $\partial A$ intersects the dividing curves on $\partial M$ $2k$ times. Now, if $N(L)$ thickened to a solid torus $N'$ with dividing slope $s$ we first note that $s$ must be between $-k$ and $-k+1$ since $L$ does not destabilize. We also note that any curve of slope between $-k$ and $-k+1$ intersects the boundary of $A$ (which has slope $\infty)$ more times than the slope $-k$ curve does. Thus, if $A$ is an annulus between $\partial M$ and $\partial N'$, then there will be a bypass for $\partial N'$ on $A$. Thus, we can thicken $N'$. We can keep doing this until we have thickened $N'$ to a solid torus with dividing slope $-k+1$, but this contradicts the fact that $L$ has the maximal contact twisting among curves isotopic to the fiber. Thus $N(L)$ does not thicken and Step 1 is complete. 

\smallskip
\noindent{\bf Step 2:} We will build a contact structure on $M$ by defining it on $N(L)$ and neighborhoods of the $A_i$ above, and we will then extend this contact structure over $S$ and show that the constructed contact structure is indeed $\xi_{std}$. 

Since $N(L)$ is a neighborhood of a Legendrian knot, there is a unique choice for the contact structure on $N(L)$ (once the characteristic foliation on $\partial N(L)$ is fixed). By considering a standard model for a Legendrian knot, we can see that $N(L)$ is fibered by Legendrian $S^1$s isotopic to $L$. Consider a neighborhood $N(A_i)$ of $A_i$. Notice that $N(A_i)\cap N(L)$ consists of two annuli, call one of them $C_i$. Smoothly, one may think of $N(A_i)$ as $C_i\times I$ where $I$ is an interval. The annulus $C_i$ may be taken to be foliated by ruling curves of $\partial N(L)$ and then we may extend the contact structure over $N(A_i)$ as an $I$-invariant contact structure $C_i\times I$. This defines a contact structure on $M$ for which all $S^1$-fibers are Legendrian and isotopic to $L$. The proof of Lemma~3.3 in \cite{Honda00b} shows that any Legendrian knot smoothly isotopic to an $S^1$-fiber has contact twisting less than or equal to $-k$. In the proof of Theorem~3.8 in \cite{Honda00b}, Honda shows how to perturb the fibration so that the fibers are all transverse to the contact structure. %{\color{red} Should we say more? Honda is not very detailed here!}

Given the boundary conditions on $\partial M=\partial S$ there are two universally tight contact structures on $S$ with this data (except when $k=1$, in which case there is one). We extend the contact structure constructed above over $S$ by either of these two universally tight contact structures. This gives a contact structure on $\#_{2g} S^1\times S^2$. We note that the same argument as the one used in the proof of Proposition~\ref{genus1} above, whose that the constructed contact structure is $\xi_{std}$. We denote the constructed solid tori by $S_k^\pm$ (when $k=1$, we omit the $\pm$). 

\smallskip
\noindent{\bf Step 3:} We are left to show that the $S_k^\pm$ are non-thickenable tori. To this end, notice that if $S_k^\pm$ thickened, then the argument at the end of Step~1 would show that $N(L)$ thickens in $M$, but we saw in Step~1 that $N(L)$ does not thicken.
\end{proof}

\begin{proof}[Proof of Theorem~\ref{non-thickenableex}]
The theorem clearly follows from the two propositions in this section which deal with the two cases in the statement of the theorem. 
\end{proof}

% references
\bibliography{references}

\def\cprime{$'$}
\begin{thebibliography}{10}

\bibitem{BakerEtnyreVanHorn-Morris12}
Kenneth~L. Baker, John~B. Etnyre, and Jeremy Van Horn-Morris.
\newblock Cabling, contact structures and mapping class monoids.
\newblock {\em J. Differential Geom.}, 90(1):1--80, 2012.

\bibitem{ChakrabortyEtnyreMin24}
Apratim Chakraborty, John~B. Etnyre, and Hyunki Min.
\newblock Cabling {L}egendrian and transverse knots.
\newblock {\em J. Differential Geom.}, 126(1):1--48, 2024.

\bibitem{ChatterjeeEtnyreMinRodewaldPre}
Rima Chatterjee, John~B. Etnyre, Hyun~Ki Min, and Thomas Rodewald.
\newblock Legendrian cable links of twist knots.
\newblock In preparation.

\bibitem{ChristianLi23}
Austin Christian and Youlin Li.
\newblock Some applications of {M}enke's {JSJ} decomposition for symplectic fillings.
\newblock {\em Trans. Amer. Math. Soc.}, 376(7):4569--4604, 2023.

\bibitem{ChristianMenkePre}
Austin Christian and Michael Menke.
\newblock A {JSJ}-type decomposition theorem for symplectic fillings, 2022.

\bibitem{ConwayEtnyreTosun21}
James Conway, John~B. Etnyre, and B\"{u}lent Tosun.
\newblock Symplectic fillings, contact surgeries, and {L}agrangian disks.
\newblock {\em Int. Math. Res. Not. IMRN}, (8):6020--6050, 2021.

\bibitem{EliashbergThurston98}
Yakov~M. Eliashberg and William~P. Thurston.
\newblock {\em Confoliations}, volume~13 of {\em University Lecture Series}.
\newblock American Mathematical Society, Providence, RI, 1998.

\bibitem{EtnyreVertesi18}
John Etnyre and Vera V\'{e}rtesi.
\newblock Legendrian satellites.
\newblock {\em Int. Math. Res. Not. IMRN}, (23):7241--7304, 2018.

\bibitem{Etnyre04b}
John~B. Etnyre.
\newblock Planar open book decompositions and contact structures.
\newblock {\em Int. Math. Res. Not.}, (79):4255--4267, 2004.

\bibitem{EtnyreHonda01b}
John~B. Etnyre and Ko~Honda.
\newblock Knots and contact geometry. {I}. {T}orus knots and the figure eight knot.
\newblock {\em J. Symplectic Geom.}, 1(1):63--120, 2001.

\bibitem{EtnyreHonda03}
John~B. Etnyre and Ko~Honda.
\newblock On connected sums and {L}egendrian knots.
\newblock {\em Adv. Math.}, 179(1):59--74, 2003.

\bibitem{EtnyreHonda05}
John~B. Etnyre and Ko~Honda.
\newblock Cabling and transverse simplicity.
\newblock {\em Ann. of Math. (2)}, 162(3):1305--1333, 2005.

\bibitem{EtnyreLafountainTosun12}
John~B. Etnyre, Douglas~J. LaFountain, and B{\"u}lent Tosun.
\newblock Legendrian and transverse cables of positive torus knots.
\newblock {\em Geom. Topol.}, 16(3):1639--1689, 2012.

\bibitem{EtnyreLiMinPre}
John~B. {Etnyre}, Youlin Li, and Hyunki Min.
\newblock Contact structures, the figure eight knot, and incompressible tori.
\newblock In preparation.

\bibitem{EtnyreMinTosunPre}
John~B. {Etnyre}, Hyun~Ki Min, Bulent Tosun, and Konstantinos Varvarezos.
\newblock Tight surgeries on torus knots.
\newblock In preparation.

\bibitem{EtnyreMinMukherjeePre}
John~B. Etnyre, Hyunki Min, and Anubhav Mukherjee.
\newblock Non-loose torus knots, 2022.

\bibitem{EtnyreRoy21}
John~B. Etnyre and Agniva Roy.
\newblock Symplectic fillings and cobordisms of lens spaces.
\newblock {\em Trans. Amer. Math. Soc.}, 374(12):8813--8867, 2021.

\bibitem{EtnyreVanHornMorris10}
John~B. Etnyre and Jeremy {Van Horn-Morris}.
\newblock Fibered transverse knots and the {B}ennequin bound.
\newblock {\em Int. Math. Res. Not.}, page~27, 2010.
\newblock arXiv:0803.0758v2.

\bibitem{GeorgeMyers20}
Whitney George and Mark Myers.
\newblock Positive twist knots and the uniform thickness property.
\newblock {\em J. Knot Theory Ramifications}, 29(8):2050058, 14, 2020.

\bibitem{Grioux93}
Emmanuel Giroux.
\newblock Topologie de contact en dimension {$3$} (autour des travaux de {Y}akov {E}liashberg).
\newblock {\em Ast\'erisque}, (216):Exp.\ No.\ 760, 3, 7--33, 1993.
\newblock S\'eminaire Bourbaki, Vol.\ 1992/93.

\bibitem{Giroux00}
Emmanuel Giroux.
\newblock Structures de contact en dimension trois et bifurcations des feuilletages de surfaces.
\newblock {\em Invent. Math.}, 141(3):615--689, 2000.

\bibitem{Gordon83}
C.~McA. Gordon.
\newblock Dehn surgery and satellite knots.
\newblock {\em Trans. Amer. Math. Soc.}, 275(2):687--708, 1983.

\bibitem{Hedden09}
Matthew Hedden.
\newblock On knot {F}loer homology and cabling. {II}.
\newblock {\em Int. Math. Res. Not. IMRN}, (12):2248--2274, 2009.

\bibitem{Hedden10}
Matthew Hedden.
\newblock Notions of positivity and the {O}zsv\'{a}th-{S}zab\'{o} concordance invariant.
\newblock {\em J. Knot Theory Ramifications}, 19(5):617--629, 2010.

\bibitem{Honda00a}
Ko~Honda.
\newblock On the classification of tight contact structures. {I}.
\newblock {\em Geom. Topol.}, 4:309--368 (electronic), 2000.

\bibitem{Honda00b}
Ko~Honda.
\newblock On the classification of tight contact structures. {II}.
\newblock {\em J. Differential Geom.}, 55(1):83--143, 2000.

\bibitem{HondaKazezMatic02}
Ko~Honda, William~H. Kazez, and Gordana Mati{\'c}.
\newblock Convex decomposition theory.
\newblock {\em Int. Math. Res. Not.}, (2):55--88, 2002.

\bibitem{LaFountain10}
Douglas~J. LaFountain.
\newblock Studying uniform thickness. {I}. {L}egendrian simple iterated torus knots.
\newblock {\em Algebr. Geom. Topol.}, 10(2):891--916, 2010.

\bibitem{LaFountain11}
Douglas~J. LaFountain.
\newblock Studying uniform thickness {II}: {T}ransversely nonsimple iterated torus knots.
\newblock {\em Algebr. Geom. Topol.}, 11(5):2741--2774, 2011.

\bibitem{LiWanPre}
Zhenkun Li and Shunyu Wan.
\newblock On legendrian representatives of non-fibered knots, 2023.

\bibitem{LidmanSivek16}
Tye Lidman and Steven Sivek.
\newblock Contact structures and reducible surgeries.
\newblock {\em Compos. Math.}, 152(1):152--186, 2016.

\bibitem{Lisca08}
Paolo Lisca.
\newblock On symplectic fillings of lens spaces.
\newblock {\em Trans. Amer. Math. Soc.}, 360(2):765--799, 2008.

\bibitem{LiscaStipsicz04}
Paolo Lisca and Andr{\'a}s~I. Stipsicz.
\newblock Ozsv\'ath-{S}zab\'o invariants and tight contact three-manifolds. {I}.
\newblock {\em Geom. Topol.}, 8:925--945 (electronic), 2004.

\bibitem{McCullough23}
Andrew McCullough.
\newblock Legendrian large cables and new phenomenon for nonuniformly thick knots.
\newblock {\em Algebr. Geom. Topol.}, 23(6):2561--2591, 2023.

\bibitem{OzsvathSzabo04a}
Peter Ozsv{\'a}th and Zolt{\'a}n Szab{\'o}.
\newblock Holomorphic disks and genus bounds.
\newblock {\em Geom. Topol.}, 8:311--334 (electronic), 2004.

\bibitem{Schonenberger2007}
Stephan Sch\"{o}nenberger.
\newblock Determining symplectic fillings from planar open books.
\newblock {\em J. Symplectic Geom.}, 5(1):19--41, 2007.

\bibitem{Torisu00}
Ichiro Torisu.
\newblock Convex contact structures and fibered links in 3-manifolds.
\newblock {\em Internat. Math. Res. Notices}, (9):441--454, 2000.

\bibitem{Yasui16}
Kouichi Yasui.
\newblock Maximal {T}hurston-{B}ennequin number and reducible {L}egendrian surgery.
\newblock {\em Compos. Math.}, 152(9):1899--1914, 2016.

\end{thebibliography}
\bibliographystyle{plain}
\end{document}